\documentclass{svjour3}
\usepackage{amssymb,latexsym,amsmath}
\usepackage{graphicx}
\numberwithin{equation}{section}

\def\A {\mathcal{A}}

\begin{document}

\title{Supersolvability and freeness for $\psi$-graphical
  arrangements\thanks{The first author would like
to thank the China Scholarship Council for
financial support. Her work was done during her visit to the
Department of Mathematics, Massachusetts Institute of Technology. The
second author was partially supported by the 
    National  Science Foundation under Grant No.~DMS-1068625.}}

\titlerunning{$\psi$-graphical arrangements}

\author{Lili Mu \and Richard P. Stanley}

\institute{Lili Mu \at Dept.\ of Mathematics\\ M.I.T.\\ Cambridge, MA
  02139\\ \emph{Present address:}   School of Mathematical Sciences\\
         Dalian University of Technology\\
         Dalian 116024, PR China\\ \email{llymu.lzp@gmail.com} \and
         Richard Stanley \at Dept.\ of 
         Mathematics E17-434\\ M.I.T.\\Ambridge, MA
         02139\\ \email{rstan@math.mit.edu}}

\date{Received: date / Accepted: date}

\maketitle
\begin{abstract}
Let $G$ be a simple graph on the vertex set $\{v_1,\dots,v_n\}$ with
edge set $E$. Let $K$ be a field.  The graphical arrangement $\A_G$
in $K^n$ is the arrangement $x_i-x_j=0, v_iv_j \in E$. An
arrangement $\A$ is 
supersolvable if the intersection lattice $L(c(\A))$ of the cone
$c(\A)$ contains a maximal chain of modular elements. The second
author has shown that a graphical arrangement $\A_G$ is supersolvable
if and only if $G$ is a chordal graph. He later considered a
generalization of graphical arrangements which are called
$\psi$-graphical arrangements. He conjectured a characterization of
the supersolvability and freeness (in the sense of Terao) of a
$\psi$-graphical arrangement. We provide a proof of the first
conjecture and state some conditions on free $\psi$-graphical
arrangements. 

\keywords{graphical arrangement\and supersolvable arrangement \and
  free arrangement\and chordal graph}
\subclass{52C35\and 05C15}
\end{abstract}

\section{Introduction}
\hspace*{\parindent}

A finite hyperplane arrangement $\A$ is a finite set of affine
hyperplanes in some vector space $V \cong K^n$, where $K$ is a
field. The \emph{intersection poset} $L(\A)$ of $\A$ is the set of all
nonempty intersections of hyperplanes in $\A$, including $V$ itself as
the intersection over the empty set, ordered by reverse
inclusion. Define the order relationship $x \leq y$ in $L(\A)$ if
$x\supseteq y$ in $V$.

Let $G$ be a graph with vertex set $V=\{v_1,\ldots,v_n\}$ and edge set
$E$. The \emph{graphical arrangement} $\A_G$ in $K^n$ is the arrangement
with hyperplane $x_i-x_j=0, v_iv_j \in E$. We will use poset notation
and terminology from \cite[Ch. 3]{4}. In particular, the intersection
poset of the graphical arrangement $\A_G$ (or of any central
arrangement) is geometric. (An arrangement $\A$ is central if
$\bigcap_{H \in \A} H \neq \emptyset$.)  Let $2^\mathbb{P}$ denote the
set of all subsets of $\mathbb{P}$, and let $\psi: V\rightarrow
2^\mathbb{P}$ satisfy $\vert \psi(v)\vert < \infty$ for all $v \in
V$. Define the $\psi$-$graphical$ $arrangement$ $\A_{G,\psi}$ to be
the arrangement in $\mathbb{R}^n$ with hyperplanes $x_i=x_j$ whenever
$v_iv_j \in E$, together with $x_i=\alpha_j$ if $\alpha_j \in
\psi(v_i)$.

In general, $\A_{G,\psi}$ is not a central arrangement and the
intersection poset $L(\A_{G,\psi})$ of $\A_{G,\psi}$ is not a
geometric lattice. Instead of $\A_{G,\psi}$ we consider the cone
$c(\A_{G,\psi})$ with coordinates $x_1,\dots,x_n,y$.  The \emph{cone
  $\psi$-graphical arrangement} $c(\A_{G,\psi})$ is the arrangement
with hyperplanes $x_i=x_j$ whenever $v_iv_j \in E$, together with
$y=0$ and $x_i=\alpha_j y$ if $\alpha_j \in \psi(v_i)$.

An element $x$ of a geometric lattice $L$ is \emph{modular} if
$\text{rk}(x)+\text{rk}(y)= \text{rk}(x\wedge y)+\text{rk}(x\vee y)$
for all $y\in L$, where rk denotes the rank function of $L$. A
geometric lattice $L$ is \emph{supersolvable} if there exists a
\emph{modular maximal chain}, i.e., a maximal chain $\hat{0}=x_0\lessdot
x_1\lessdot \cdots \lessdot x_n=\hat{1}$ such that each $x_i$ is
modular. A central arrangement $\A$ is \emph{supersolvable} if its
intersection lattice $L(\A)$ is supersolvable.

A graph is \emph{chordal} if each of its cycles of four or more
vertices has a chord, which is an edge that is not part of the cycle
but connects two vertices of the cycle. Equivalently, every induced
cycle in the graph should have exactly three vertices.  A graphical
arrangement $\A_G$ is supersolvable if and only if $G$ is a chordal
graph \cite[Cor. 4.10]{1}.

It is well known that the elements $X_\pi$ of $L(\A_G)$ correspond to
the connected partitions $\pi$ of $V(G)$, i.e., the partitions
$\pi=\{B_1,\ldots,B_k\}$ of $V(G)$ such that the restriction of $G$ to
each block $B_i$ is connected.

We have $X_\pi \leq X_\sigma$ in $L(\A_G)$ if and only if every block
of $\pi$ is contained in a block of $\sigma$. Hence $L(\A_G)$ is
isomorphic to an induced subposet $L_G$ of $\Pi_n$, the lattice of
partitions of the set $\{1,2,\ldots,n\}$. From the definition of
$L(c(\A_{G,\psi}))$ it is easy to see that $L(\A_G)$ is an interval of
$L(c(\A_{G,\psi}))$, namely, the interval from the bottom element
$\hat{0}$ (the ambient space $K^n$) to the intersection of all the
hyperplanes $x_i=x_j$ of $c(\A_{G,\psi})$. 
For brevity, an element
 $$ X_{\sigma}=(x_1,\ldots,x_{i-1},\alpha_i
      y,x_{i+1},\ldots,x_{j-1},\alpha_i y,x_{j+1},\ldots,x_n,y) $$
($\alpha_i \in \psi(v_i)$ or $\alpha_i \in \psi(v_j)$) of
$L(c(\A_{G,\psi}))$ is written as $\sigma : v_i=v_j=\alpha_i y$, or
more briefly as $\sigma=\{v_iv_j\alpha_i y\}$, and an element
 $$ X_{\delta}=(x_1,\ldots,x_{i-1},0,x_{i+1},\ldots,
  x_{j-1},0,x_{j+1},\ldots,x_n,0) $$
is written as $\delta: v_i=v_j=y=0$, or more briefly as
$\delta=\{v_iv_jy0\}$.  The following sufficient condition for the
supersolvability of a $\psi$-graphical arrangement is stated in
\cite{2} without proof.

\begin{theorem}\label{CH}
Let $(G,\psi)$ be as above. Suppose that we can order the vertices of
$G$ as $v_1,v_2,\ldots,v_n$ such that:

$(1)$ $v_{i+1}$ connects to previous vertices along a clique (so by
Lemma~\ref{AA} below $G$ is chordal).

$(2)$ If $i<j$ and $v_i$ is adjacent to $v_j$, then
$\psi(v_j)\subseteq \psi(v_i)$.

Then $\A_{G,\psi}$ is supersolvable.
\end{theorem}

\begin{proof}
To prove that $\A_{G,\psi}$ is supersolvable we need to find a modular
maximal chain in $L(c(\A_{G,\psi}))$. 
We will show that a modular maximal chain is given by $\hat{0} < \pi_1
< \cdots < \pi_n < \hat{1}$, where $\pi_i=\{v_1v_2\cdots
v_{i-1}y0\}$. 
First we prove that $\pi_n=\{v_1v_2\cdots v_{n-1}y0\}$ is a modular
element. 
For any $\sigma=\{B_1,B_2,\ldots, B_t\} \in L(c(\A_{G,\psi}))$, we only
need to consider the block $B_i$ which contains $v_n$. 
If $B_i=\{v_n\}$ then $\sigma < \pi_n$. Hence
$\text{rk}(\pi_n)+\text{rk}(\sigma)=\text{rk}(\pi_n \wedge \sigma)+
\text{rk}(\pi_n \vee \sigma)$. 

If $B_i=\{v_{i_1}\cdots v_{i_m}v_n\}$ then $\pi_n \vee \sigma =
\hat{1}$. Since $v_{n}$ connects to previous vertices along a clique,
the block $B_i'=\{v_{i_1}\cdots v_{i_m}\}$ exists. Then 
$\pi_n\wedge \sigma = \{B_1,\ldots,B_{i-1},B_{i+1},$ $\ldots,
B_t,B_i',v_n\}$. 
Hence $\text{rk}(\pi_n \wedge \sigma )= \text{rk}(\sigma)-1$ and
$\text{rk}(\pi_n)+\text{rk}(\sigma)=\text{rk}(\pi_n \wedge \sigma)+
\text{rk}(\pi_n \vee \sigma)$. 

If $B_i=\{v_{i_1}\cdots v_{i_m}v_ny0\}$ then $\pi_n \vee \sigma =
\hat{1}$ and   
$$\pi_n\wedge \sigma = \{B_1,\ldots,B_{i-1},B_{i+1},\dots,
B_t,v_{i_1}\cdots v_{i_m}y0,v_n\}.$$  
Hence $\text{rk}(\pi_n \wedge \sigma )= \text{rk}(\sigma)-1$ and
$\text{rk}(\pi_n)+\text{rk}(\sigma)=\text{rk}(\pi_n \wedge \sigma)+
\text{rk}(\pi_n \vee \sigma)$. 

If $B_i=\{v_{i_1}\cdots v_{i_m}v_n\alpha_j y\}$ $(\alpha_j \in
\psi(v_{i_j}), 1\leq j\leq m,$ or $\alpha_j \in \psi(v_{n}))$ then
$\pi_n\vee \sigma = \hat{1}$. Since  $\psi(v_n)\subseteq
\psi(v_{i_j})$ if $v_{i_j}v_n\in E$, i.e., if $\alpha_j \in
\psi(v_{n})$, we have $\alpha_j \in \psi(v_{i_j})$. 
 Hence the block $B_i'=\{v_{i_1}\cdots v_{i_m}\alpha_j y\}$ exists and
$\pi_n \wedge \sigma = \{B_1,\dots,B_{i-1},B_{i+1},\dots,
 B_t,B_i',v_n\}$. Then $\text{rk}(\pi_n \wedge \sigma )=
 \text{rk}(\sigma)-1$ and
 $\text{rk}(\pi_n)+\text{rk}(\sigma)=\text{rk}(\pi_n \wedge \sigma)+
 \text{rk}(\pi_n \vee \sigma)$. 

Hence we get that $\pi_n=\{v_1v_2\cdots v_{n-1}y0\}$ is a modular
element.  Now if $\pi_{n-1}=\{v_1v_2\cdots v_{n-2}y0\}$ is modular in
the interval $[\hat{0}, \pi_n]$, then it is modular in
$L(c(\A_{G,\psi}))$ \cite[Prop. 4.10(b)]{1}.  Therefore we just need to
show that $\pi_{n-1}$ is modular in the interval $[\hat{0}, \pi_n]$.

Since all elements $\sigma$ in $[\hat{0},\pi_n]$ must satisfy that
$\sigma$ has a block $B_i=\{v_n\}$, we can ignore the block
$B_i=\{v_n\}$. In the same way we can get that
$\pi_{n-1}=\{v_1v_2\cdots v_{n-2}y0\}$ is a modular element in the
interval $[\hat{0}, \pi_n]$. Continuing the procedure, we get the
modular maximal chain $\hat{0} < \pi_1 < \cdots < \pi_n < \hat{1}$.
\end{proof}

Our main result is the converse to Theorem 1.
\begin{theorem}\label{CHS}
The sufficient condition in Theorem \ref{CH} for the supersolvability
of $\A_{G,\psi}$ is also necessary. 
\end{theorem}

 Before we prove Theorem 2, the following two results of Dirac \cite{7}
 are required. A vertex is \emph{simplicial} in a graph if its neighbors
 form a complete subgraph.  A graph is \emph{recursively simplicial}
 if it consists of a single vertex, or if 
 it contains a simplicial vertex $v$ and when $v$ is removed the
 subgraph that remains is recursively simplicial. It is well-known and
 easy to see that if $G$ is recursively simplicial and $v$ is
 \emph{any} vertex, then $G-v$ is recursively simplicial

\begin{lemma} \label{AA}
$G$ is chordal if and only if $G$ is recursively simplicial.
\end{lemma}

\begin{lemma} \label{AB}
Every chordal graph $G$ that is not a complete graph has at least two
non-adjacent simplicial vertices.
\end{lemma}

\begin{proof}[of Theorem \ref{CHS}]
Condition (1) is easy to check, because $L(\A_{G})$ is an interval of
$L(c(\A_{G,\psi}))$.  Since intervals of supersolvable lattices are
supersolvable (\cite[Prop. 3.2]{3}), we have that $L(c(\A_G))$ is
supersolvable. Hence by \cite[Prop. 2.8]{3} $G$ is chordal.

By Lemma \ref{AB} we know that there are at least two nonadjacent
simplicial vertices in the chordal graph $G$. Suppose that there is a
simplicial vertex, say $v_{i_n}$, which satisfies the following
condition: 
 \begin{equation} \psi(v_{i_n}) \subseteq \psi(v_{i_j}) \text{ for all }
v_{i_j}v_{i_n} \in E. \label{star} \end{equation}
 Then we label  $v_{i_n}$ as $v_n$ and remove this vertex. By
 Lemma \ref{AA} we know that the remaining graph is still recursively
 simplicial. Continuing in this way, suppose that there is a
 simplicial vertex, which we label as $v_{n-1}$ and then remove it. 
 Continue this procedure. If condition (2) is not necessary then
 that means there exists one step $m$ in the above procedure such that
 all remaining simplicial vertices do not satisfy the condition
 \eqref{star}. Then we will show that there is no modular maximal
 chain in $L(c(\A_{G,\psi}))$. 
 
Next, we show that among all the coatoms only $\sigma_i
=\{v_{i_1}v_{i_2}\cdots v_{i_{n-1}}y0\}$ and $\delta_i =
\{v_1v_2\cdots v_n\alpha_i y\}$, $\alpha_i \in \psi(v_i), 1\leq i\leq
n$ could be modular elements of $L(c(\A_{G,\psi}))$. We claim that a
coatom is not modular if it has more than two blocks or it has two
blocks but the cardinalities of both of the blocks are greater than
$1$.

First, it is easy to check that any coatom $\sigma$ is not modular if
it has more than two blocks. Suppose $\sigma=\{A,B,C\}$ is a
coatom. Since $\text{rk}(\sigma)= n-1$, i.e., $\dim(\sigma)=1$,
$A,B$ and $C$ can only be $\{v_{i_1}v_{i_2}\cdots v_{i_{j_i}}\alpha_i
y\}$ where $i=1,2,3$ and $\alpha_i \in \psi(v_{i_m}),
m=1,2,\ldots,j_i$. Then $\gamma = \{v_1v_2\cdots v_n,y0\}$,
$\text{rk}(\sigma)=\text{rk}(\gamma)=n-1$, $\text{rk}(\sigma \vee
\gamma)= n$ but $\text{rk}(\sigma \wedge \gamma)< n-2.$ Hence $\sigma$
is not modular.

Moreover if $\sigma = \{A,B\}$ is a coatom such that $ \vert A
\vert>1$ and $ \vert B\vert>1$ then $\sigma$ is also not
modular. Without loss of generality assume that there exist $u, v \in
A$ and $u',v' \in B$ such that $u\neq v'$ and $u'\neq v$, $uu' \in
E(G)$, and $vv' \in E(G)$. Let $\gamma = \{(A\cup u')\backslash v, (B\cup
v)\backslash u'\}$. Then $\text{rk}(\sigma)=\text{rk}(\gamma)=n-1,
\text{rk}(\sigma \vee \gamma)= n$ but $\text{rk}(\sigma \wedge
\gamma)< n-2.$ Hence $\sigma$ is not modular.

Therefore, among all coatoms, only $\sigma_i =\{v_{i_1}v_{i_2}\cdots
v_{i_{n-1}}y0\}$ and 
  $$ \delta_i = \{v_1v_2\cdots v_n \alpha_i y,
           \alpha_i \in \psi(v_i),1\leq i\leq n\} $$ 
could be modular elements. Similarly, among all the elements which
$\sigma_i$ covers, only $\{(v_{i_1}v_{i_2}\cdots
v_{i_{n-1}}\backslash v_{i_j})y0\}$ could be modular.

If $v_{i_n}$ is not a simplicial vertex then we show
that $\sigma_i$ is not modular. Without loss of generality assume
$v_{i_s}v_{i_n} \in E$ and $v_{i_t}v_{i_n} \in E$ but
$v_{i_s}v_{i_t}\notin E$. 
Let $\gamma=\{ (v_{i_1}v_{i_2}\ldots v_{i_{n-1}}\backslash v_{i_s}v_{i_t})y0,
v_{i_s}v_{i_t}v_{i_n}\} $. Then
$\text{rk}(\sigma)=\text{rk}(\gamma)=n-1, \text{rk}(\sigma \vee
\gamma)= n$ but  
$\text{rk}(\sigma \wedge \gamma)< n-2.$ Hence $\sigma_i$ is not
modular if $v_{i_n}$ is not a simplicial vertex. 

 We now show that if $v_{i_n}$ is a simplicial vertex but does not
 satisfy condition \eqref{star}, then $\sigma_i$ is not modular.  Without
 loss of generality assume that $\alpha_i \in \psi(v_{i_n})$ but
 $\alpha_i \notin \psi(v_{i_j})$ for $v_{i_j}v_{i_n} \in E$. Then
 $\gamma =\{ v_{i_j}v_{i_n}\alpha_i y\}$, $\text{rk}(\sigma_i)=n-1,
 \text{rk}(\gamma)=2, \text{rk}(\sigma_i \vee \gamma)= n$ but
 $\text{rk}(\sigma_i \wedge \gamma)=0.$ From the above discussion, if
 condition (2) is not necessary then there exists one step $m$ such
 that all remaining simplicial vertices do not satisfy condition
 \eqref{star}. That means that all $\{v_{i_1}v_{i_2}\cdots
 v_{i_{n-m}}y0\}$ are not modular elements. Hence there is no modular
 maximal chain from $\hat{0}$ to $\sigma_i$.

We now show that if $\delta_i$ is modular then $\alpha_i \in
\psi(v_i)$ for all $i \in [n]$.  Equivalently, we show that if there
exists some $v_m$ such that $\alpha_i \notin \psi(v_m)$, then
$\delta_i$ is not modular.  Let $\gamma = \{v_1v_2\cdots v_n\backslash v_m, v_m
y0\}$. Hence $\text{rk}(\delta_i)=\text{rk}(\gamma)=n-1,
\text{rk}(\delta_i\vee \gamma)= n$ but $\text{rk}(\delta_i \wedge
\gamma)< n-2.$ From the above discussion, if condition (2) is not
necessary then there are at least two nonadjacent simplicial vertices,
say $v_s$ and $v_t$, which do not satisfy condition \eqref{star}. It
means that there exist $\alpha_s \in \psi(v_s), \alpha_t \in
\psi(v_t)$ and $\alpha_s,\alpha_t \neq \alpha_i$. If
$\alpha_s=\alpha_t$ then let $ \gamma = \{(v_1v_2\cdots
v_n\backslash v_{s}v_{t}) \alpha_i y, v_{s} v_{t} \alpha_s y\}$.  Hence
$\text{rk}(\delta_i)=\text{rk}(\gamma)=n-1, \text{rk}(\delta_i \vee
\gamma)= n$ but $\text{rk}(\delta \wedge \gamma)< n-2$, so $\delta_i$
is not modular.

If $\alpha_s \neq \alpha_t$ then  let $\gamma=\{v_t\alpha_t y,v_s
\alpha_s y, (v_1v_2\cdots v_n\backslash {v_t v_s}) \alpha_i y\}$. 
Now $\text{rk}(\delta_i)=\text{rk}(\gamma)=n-1, \text{rk}(\delta_i
\vee \gamma)= n$ but $\text{rk}(\delta_i \wedge \gamma)< n-2$, so
$\delta_i$ is not modular. 

Therefore if there does not exist a labeling such that conditions (1)
and (2) holds, then we can't find a modular maximal chain in
$L(c(\A_G))$. Hence the proof is complete. 
\end{proof}

We call $v_1,\ldots,v_n$ a \emph{vertex elimination order} for $G$ if 
$v_{i+1}$ connects to previous vertices along a clique. For any
supersolvable arrangement $\A$ of rank $n$ the characteristic
polynomial of $\A$ 
(defined, e.g., in \cite[{\S}1.3]{1} or \cite[{\S}3.11.2]{4}) factors
as $\chi_G(q) = \prod_{i=1}^n(q-a_i)$, where $a_1,\dots,a_n$ are
nonnegative integers, called the \emph{exponents} of $\A$. 
There is a simple combinatorical interpretation of the exponents of
$\A(G)$ when $G$ is chordal.  
\begin{proposition}\cite[Lemma 3.4]{9}
Let $G$ be a chordal graph with vertex elimination order
$\{v_1,\ldots,v_n\}$. For $1\leqslant i\leqslant n$ let $b_i$ be the
degree of $v_i$ in the graph $G-\{v_n,\ldots,v_{i+1}\}$. Then
$\{b_1,\ldots,b_{n}\}$ are the exponents of the supersolvable
arrangement $\A(G)$. 
\end{proposition}

It is not hard to get a similar property for the supersolvable
arrangement $\A_{G,\psi}$. We omit the proof of this proposition. 

\begin{proposition}
Let $(G,\psi)$ be a chordal graph with vertex elimination order
$\{v_1,\ldots,v_n\}$. Assume that for any $v_iv_j\in E(G)$ such that
if $i<j$, we have $\psi(v_j)\subseteq \psi(v_i)$. For $1\leqslant
i\leqslant n$ let $b_i$ be the sum of $\vert \psi(v_i)\vert$ and the
degree of $v_i$ in the graph $G-\{v_n,\ldots,v_{i+1}\}$. Then
$\{b_1,\ldots,b_{n}\}$ are the exponents of the supersolavable
arrangement $\A_{G,\psi}$.
\end{proposition}

There is another conjecture in \cite{2}. It is well known that every
supersolvable arrangement is free (in the sense of Terao \cite[{\S}6.3]{6})
and every free graphical arrangement is supersolvable. Thus the second
author proposed the following conjecture.

\begin{conjecture}\label{b} 
If $\A_{G,\psi}$ is a free $\psi$-graphical arrangement, then
$\A_{G,\psi}$ is supersolvable. 
\end{conjecture}

We  are unable to prove this conjecture, but we do have the following
weaker result, which we simply state without proof. The proof involves
the inheritance of freeness under localization of arrangements and a
result of Yoshinaga \cite{8} on the freeness of 3-arrangements. 

\begin{theorem} \label{THI}
 The $\psi$-graphical arrangement $\A_{G,\psi}$ is not free if there
 is an edge $v_iv_j\in E(G)$ such that $\psi(v_i)\nsubseteq \psi(v_j)$
 and $\psi(v_j)\nsubseteq \psi(v_i)$. 
\end{theorem}

\end{document}